\documentclass{birkjour}

\usepackage[utf8]{inputenc}
\usepackage[T1]{fontenc}
\usepackage[english]{babel}
\usepackage{mathtools,amsmath,amssymb,amsfonts,amsthm,mathrsfs}
\mathtoolsset{showonlyrefs}

\usepackage{geometry}
\usepackage{graphicx,color}
\usepackage{tikz,pgfplots,pgf}
\usepackage[caption=false]{epsfig}
\usepackage{subcaption}
\usepackage[font=small]{caption}
\usepackage[pdftex]{hyperref}

\newtheorem{theorem}{Theorem}[section]

\newtheorem{corollary}{Corollary}[section]

\theoremstyle{definition}
\newtheorem{remark}{Remark}[section]

\numberwithin{equation}{section}

\begin{document}

\title[Prescribed energy periodic solutions of relativistic Kepler problems]{Prescribed energy periodic solutions \\of Kepler problems \\with relativistic corrections}

\author[A.~Boscaggin]{Alberto Boscaggin}

\address{
Department of Mathematics ``Giuseppe Peano'', University of Torino\\
Via Carlo Alberto 10, 10123 Torino, Italy}

\email{alberto.boscaggin@unito.it}

\author[W.~Dambrosio]{Walter Dambrosio}

\address{
Department of Mathematics ``Giuseppe Peano'', University of Torino\\
Via Carlo Alberto 10, 10123 Torino, Italy}

\email{walter.dambrosio@unito.it}

\author[G.~Feltrin]{Guglielmo Feltrin}

\address{
Department of Mathematics, Computer Science and Physics, University of Udine\\
Via delle Scienze 206, 33100 Udine, Italy}

\email{guglielmo.feltrin@uniud.it}

\thanks{Work written under the auspices of the Grup\-po Na\-zio\-na\-le per l'Anali\-si Ma\-te\-ma\-ti\-ca, la Pro\-ba\-bi\-li\-t\`{a} e le lo\-ro Appli\-ca\-zio\-ni (GNAMPA) of the Isti\-tu\-to Na\-zio\-na\-le di Al\-ta Ma\-te\-ma\-ti\-ca (INdAM).
\\
\textbf{Preprint -- March 2023}} 



\date{}

\dedicatory{}

\begin{abstract}
We consider two different relativistic versions of the Kepler problem in the plane: the first one involves the relativistic differential operator, the second one involves a correction for the usual gravitational potential due to Levi-Civita.
When a small external perturbation is added into such equations, we investigate the existence of periodic solutions with prescribed energy bifurcating from periodic invariant tori of the unperturbed problems. Our main tool is an abstract bifurcation theory from periodic manifolds developed by Weinstein, which is applied in the case of nearly integrable Hamiltonian systems satisfying the usual KAM isoenergetic non-degeneracy condition.
\end{abstract}

\maketitle

\section{Introduction}\label{section-1}

The Kepler problem of classical mechanics consists in the study of the motion of a non-attracting body of mass $m$ in the gravitational field of a second fixed body. According to Newton's law, it can thus be described by the equation
\begin{equation}\label{eq-kep}
m \ddot x = - \alpha \frac{x}{\vert x \vert^3}, \qquad x \in \mathbb{R}^d \setminus \{0\},
\end{equation}
where $\alpha = G M m$, with $G$ the gravitational constant and $M$ the mass of the second body (placed at the origin, for simplicity), $d\in\mathbb{N}\setminus\{0\}$, see for instance \cite{Co-03, Po-76}.

With the aim of including relativistic effects in the above problem, some variants of equation \eqref{eq-kep} have been proposed.
A first one, theoretically framed in the context of special relativity, comes from the use of the relativistic kinetic energy
\begin{equation*}
K(\dot x) = mc^2 \bigl{(} 1 - \sqrt{1-\vert \dot x \vert^2/c^2} \bigr{)}
\end{equation*}
instead of the usual one, where $c > 0$ is the speed of light (notice that
$K(\dot x) = \tfrac{1}{2} m \vert \dot x \vert^2 + o(c^{-2})$ for $c \to +\infty$); with no modifications for the Newtonian gravitational energy, the corresponding Euler--Lagrange equation thus reads as
\begin{equation}\label{eq-keprelintro}
\dfrac{\mathrm{d}}{\mathrm{d}t}\left(\dfrac{m\dot{x}}{\sqrt{1-\vert \dot{x}\vert ^2/c^2}}\right)
= -\alpha \, \dfrac{x}{\vert x\vert ^3}, \qquad x \in \mathbb{R}^d \setminus \{0\},
\end{equation}
see, among others, \cite{AnBa-71,Bo-04,Ji-13,LeMo-PP,MuPa-06}. A second relativistic variant of equation \eqref{eq-kep} arises, instead, by keeping the usual kinetic energy but imposing a correction for the gravitational potential: precisely
\begin{equation*}
V(x) = -\frac{\kappa(c)}{\vert x \vert} - \frac{\lambda(c)}{\vert x \vert^2},
\end{equation*}
where $\kappa(c)$ and $\lambda(c)$ are suitable positive constant with the property that $\kappa(c) \to \alpha$ and $\lambda(c) \to 0$ as $c \to +\infty$. 
The reason for this ad-hoc correction of the Newtonian potential, suggested at first by Levi-Civita \cite{LC-28}, lies in the fact that the corresponding equation
\begin{equation}\label{eq-lcintro}
m \ddot x = -\kappa(c) \dfrac{x}{\vert x\vert ^{3}} - 2\lambda(c) \dfrac{x}{\vert x\vert ^4}, \qquad x \in \mathbb{R}^d \setminus \{0\},
\end{equation}
turns out to be consistent with the theory of general relativity, see for instance \cite[Chapter~12.3~B]{GoPoSa-02}.


Besides being well recognized in the physics community, equations \eqref{eq-keprelintro} and \eqref{eq-lcintro} have attracted the attention of mathematicians from dynamical systems and nonlinear analysis: the main problem in this context is to understand which aspects of the unperturbed dynamics are preserved if some coefficient is allowed to vary with time and/or if some external force is assumed to act on the body in addition to the Newtonian gravity. Contributions in this direction dealing with equation \eqref{eq-lcintro} are more classical and can be found, among others, in \cite{AmBe-90, FoGa-17, LaLlNu-91, NuCaLl-91}. On the other hand, probably due to the nonlinear nature of the involved differential operator \cite{Ma-13}, the interest for equation \eqref{eq-keprelintro} seems to be more recent and we can quote the papers \cite{BoDaFe-21, BoDaPa-23, BoDaPa-pp, ToUrZa-13, Za-13}.

Following this line of research, in this paper, we consider, in dimension $d=2$, the following perturbed versions of the above equations,
\begin{equation}\label{eq-pert1}
\dfrac{\mathrm{d}}{\mathrm{d}t}\left(\dfrac{m\dot{x}}{\sqrt{1-\vert \dot{x}\vert ^2/c^2}}\right)
= -\alpha\, \dfrac{x}{\vert x\vert ^3} + \varepsilon \, \nabla U(x), \qquad x \in \mathbb{R}^2 \setminus \{0\},
\end{equation}
and
\begin{equation}\label{eq-pert2}
m \ddot x = -\kappa(c) \dfrac{x}{\vert x\vert ^{3}} - 2\lambda(c) \dfrac{x}{\vert x\vert ^4} + \varepsilon \, \nabla U(x), \qquad x \in \mathbb{R}^2 \setminus \{0\},
\end{equation}
where $\varepsilon$ is a small real parameter and $U$ is a $\mathcal{C}^{\infty}$ external potential.
In this setting, our interest is in investigating the existence of periodic solutions bifurcating from non-circular periodic solutions of the corresponding unperturbed problem ($\varepsilon = 0$). Let us notice that a strictly related problem was recently considered in \cite{BoDaFe-21, BoDaFe-pp}: however, while therein the bifurcating solutions were required to have \textit{prescribed period} $T > 0$, here we focus on the complementary, and more delicate, case of solutions with \emph{prescribed energy} $H \in \mathbb{R}$. This means that
\begin{equation*}
\frac{1}{2}m \vert \dot x(t) \vert^2 - \frac{\kappa(c)}{\vert x(t) \vert} - \frac{\lambda(c)}{\vert x(t) \vert^2} - \varepsilon \, U(x(t)) = H,
\end{equation*} 
when $x(t)$ is a solution of the more classical \eqref{eq-pert2}; on the other hand, in view of the Lagrangian structure of \eqref{eq-pert1}, it means that
\begin{equation*}
mc^2\left( \frac{1}{\sqrt{1-\vert \dot{x}\vert ^2/c^2}}-1\right)-\dfrac{\alpha}{\vert x\vert }-\varepsilon \, U(x) = H,
\end{equation*} 
whenever $x(t)$ solves \eqref{eq-pert1}. Notice that, in principle, the values of $H$ carrying non-circular periodic solutions are not known and must be characterized by phase-plane analysis arguments.

After having performed this preliminary step, the main tool for the proof of our result is an abstract bifurcation theory from periodic manifolds, developed by A.~Weinstein in a series of papers \cite{We-73, We-77, We-78}. Roughly speaking, this theory ensures that, whenever there is a non-degenerate manifold $\Sigma$ of periodic solutions with energy $H$ for a Hamiltonian system $\dot z = X_{\mathcal{K}_{0}} (z)$ on a symplectic manifold $M$,
then a finite number of periodic solutions, with energy $\mathcal{H}_\varepsilon = H$, exist for the system
\begin{equation*}
\dot z = X_{\mathcal{H}_\varepsilon}(z)
\end{equation*} 
whenever the perturbed Hamiltonian $\mathcal{H}_\varepsilon$ is close enough to the unperturbed one $\mathcal{K}_{0}$. We point out that the notion of non-degeneracy, in this fixed-energy context, is quite delicate and several non-equivalent notions have been considered in the literature, see Remark~\ref{rem-nondeg}.

In our paper, Weinstein's theory is combined with tools from Hamiltonian dynamics. Indeed, both equations \eqref{eq-keprelintro} and \eqref{eq-lcintro} can be regarded as Hamiltonian systems with two degrees of freedom $x$ and $p$, where the momentum $p$ is given by 
$p = m\dot x$ in the more classical case of equation \eqref{eq-lcintro}, while $p = m\dot x/\sqrt{1-\vert \dot x \vert^2/c^2}$ for equation 
\eqref{eq-keprelintro}. As a consequence of the invariances by time-translation and space-rotations, two independent integrals can be found: the energy and the angular momentum $x \wedge p$.
Thus, assuming that a periodic solution $x$ with energy $H$ is found, the periodic manifold $\Sigma$ for the application of the abstract bifurcation theory is
nothing but the two-dimensional torus containing $x$ together with all its time-translations and space-rotations. After passing to action-angle coordinates, the non-degeneracy condition becomes the standard non-degeneracy condition used in the isoenergetic version of KAM theory, see for instance \cite{Ar-89, BrHu-91}. By explicit computations, such a condition is easily checked to hold true both in the case of \eqref{eq-keprelintro} and \eqref{eq-lcintro}, thus concluding the proof.

\smallskip

The plan of the paper is the following. In Section~\ref{section-2}, we review Weinstein's theory, with a focus to the case of nearly integrable Hamiltonian systems, needed for our application. Then, in Section~\ref{sec-autonomo} we give the statement and the proof of our main results for equations \eqref{eq-pert1} and \eqref{eq-pert2}.

\section{Preliminaries}\label{section-2}

In this section, we state and prove a result ensuring, for a nearly integrable Hamiltonian system, bifurcation of prescribed energy periodic solutions from periodic tori of the integrable unperturbed problem. This result is obtained as a corollary of a general bifurcation theory from periodic manifolds by A.~Weinstein \cite{We-73, We-77, We-78}. For the reader's convenience, we first briefly review this theory in Section~\ref{sec2.1}, then in Section~\ref{sec2.2} we provide our result (Theorem~\ref{teo-weinstein}).

\subsection{An abstract bifurcation result by Weinstein}\label{sec2.1}

Let $(M,\omega)$ be a symplectic manifold. Given $\mathcal{K}_{0} \in \mathcal{C}^{\infty}(M)$ we consider the Hamiltonian system
\begin{equation}\label{hs}
\dot z = X_{\mathcal{K}_{0}}(z),
\end{equation} 
where as usual $X_{\mathcal{K}_{0}}$ denotes the Hamiltonian vector field associated to  $\mathcal{K}_{0}$ (that is, $\omega(X_{\mathcal{K}_{0}},Y) = \mathrm{d}\mathcal{K}_{0}(Y)$ for every vector field $Y$ on $M$). For every $\xi \in M$, the unique (local) solution of \eqref{hs} satisfying $z(0) = \xi$ will be denoted by 
$z(t;\xi)$; as well known, the value of 
the Hamiltonian $\mathcal{K}_{0}$ is constant along $z(t;\xi)$, that is
\begin{equation*}
\mathcal{K}_{0}(z(t;\xi)) = H, \quad \text{for every $t$,}
\end{equation*} 
where $H = \mathcal{K}_{0}(\xi)$.

We are interested in values $H$ such that the level set $\mathcal{K}_{0}^{-1}(H)$ contains non-constant periodic solutions of \eqref{hs}.
More precisely, let us define the set $\mathrm{Per}_{\mathcal{K}_{0}}^H$ of periodic points with energy $H$
as the set of couples $(\xi,\tau) \in M \times (0,+\infty)$ such that $z(t;\xi)$ is a non-constant $\tau$-periodic solution of \eqref{hs} with energy $H$, that is
\begin{equation*}
z(\tau;\xi) = \xi, \qquad \mathcal{K}_{0}(\xi) = H, \qquad X_{\mathcal{K}_{0}}(\xi) \neq 0.
\end{equation*}
Incidentally, notice that $\tau$ does not need to be the minimal period and, actually, 
if $(\xi,\tau)$ belongs to $\mathrm{Per}_{\mathcal{K}_{0}}^H$ then $(\xi,N\tau)$ belongs to $\mathrm{Per}_{\mathcal{K}_{0}}^H$ as well, for every integer $N \geq 1$. 
Moreover, of course, if $(\xi,\tau)$ belongs to $\mathrm{Per}_{\mathcal{K}_{0}}^H$ then the one-dimensional torus $\{(z(t;\xi),\tau) \colon t \in [0,\tau)\}$ belongs to $\mathrm{Per}_{\mathcal{K}_{0}}^H$ as well.

We say that a set $\Sigma \subset \mathrm{Per}_{\mathcal{K}_{0}}^H$ is a \textit{periodic manifold} if it satisfies the following conditions:
\begin{itemize}
\item[$(i)$] $\Sigma$ is a closed submanifold of $M \times \mathbb{R}$,
\item[$(ii)$] the restriction to $\Sigma$ of the projection $\pi \colon M \times \mathbb{R} \to M$ is an embedding.
\end{itemize}
Notice that, roughly, condition $(ii)$ ensures that for a periodic point $(\xi,\tau) \in \Sigma$, the period has to be uniquely chosen (thought not necessarily the minimal one and, in general, depending, smoothly, on $\xi$).

The crucial additional assumption to be imposed in order to ensure bifurcation of periodic solutions (with energy $H$) from the periodic manifold 
$\Sigma$ is a non-degeneracy condition. To introduce it, we consider, for any $(\xi,\tau) \in \Sigma$, the so-called monodromy operator at $\xi$, that is, the linear operator $P \colon T_\xi M \to T_\xi M$ given by
\begin{equation*}
P \eta =\partial_{\zeta} z(\tau; \zeta)\vert _{\zeta = \xi} \,\eta. 
\end{equation*}
It is well known that both $X_{\mathcal{K}_{0}}(\xi)$ and $T_\xi (\mathcal{K}_{0}^{-1}(H))$ are invariant under $P$; moreover, $X_H(\xi) \in T_\xi (\mathcal{K}_{0}^{-1}(H))$
(incidentally, let us observe that, since $X_{\mathcal{K}_{0}}(\xi) \neq 0$, the level set $\mathcal{K}_{0}^{-1}(H)$ is a submanifold around $\xi$ and thus the tangent space $T_\xi (\mathcal{K}_{0}^{-1}(H))$ is well-defined).
With this in mind, a periodic manifold $\Sigma$ is said to be \emph{non-degenerate}  if the following condition holds:
\begin{itemize}
\item[$(iii)$] for every $(\xi,\tau) \in \Sigma$ and for every $\eta \in T_\xi (\mathcal{K}_{0}^{-1}(H))$ it holds that
\begin{equation*}
\eta \in T_\xi \pi(\Sigma) \quad \Longleftrightarrow \quad \eta = P \eta + \lambda X_H(\xi), \quad \mbox{ for some } \lambda \in \mathbb{R},
\end{equation*}
\end{itemize}
(cf.~\cite[Lemma~1.1]{We-73}).

\begin{remark}\label{rem-nondeg}
The above non-degeneracy condition might seem quite mysterious at first sight, since in principle one could expect 
the eigenspace $\mathcal{E'} \subset T_\xi (\mathcal{K}_{0}^{-1}(H)) $ of fixed points of $P$ to play a role.
The reason why the larger space 
\begin{equation*}
\mathcal{E} =  \bigl{\{} \eta \in T_\xi (\mathcal{K}_{0}^{-1}(H)) \colon \eta = P \eta + \lambda X_H(\xi), \mbox{ for some } \lambda \in \mathbb{R} \bigr{\}} \supset \mathcal{E}' 
\end{equation*}
has to be considered lies in the fact that we are dealing with periodic solutions with \emph{unprescribed period}.
Indeed, if $(\xi(s),\tau(s))$ is a smooth curve in $\Sigma$ with $(\xi(0),\tau(0)) = (\xi,\tau)$, then differentiating with respect to $s$ and taking $s = 0$ yields
\begin{equation*}
\xi'(0) = P \xi'(0) + \tau'(0)X_H(\xi(0)),
\end{equation*} 
where, in general, $\tau'(0)$ does not need to be zero. 
This shows that 
\begin{equation*}
T_\xi \pi(\Sigma) \subset \mathcal{E}
\end{equation*}
and thus the non-degeneracy conditions amounts in requiring that all vectors in $\mathcal{E}$ belongs to $T_\xi \pi(\Sigma)$.
We stress that, on the other hand, the fact that periodic solutions are required to have \emph{prescribed energy}
leads to consider the monodromy operator $P$ not on the whole $T_\xi M$, but on its invariant subspace $T_\xi (\mathcal{K}_{0}^{-1}(H))$.

We point out, however, that other non-degeneracy conditions for the periodic problem with prescribed energy have been considered in the literature. For instance, in \cite{We-77} the additional condition that the algebraic multiplicity of $1$ as an eigenvalue of $P$ is equal to the dimension of $\Sigma$ is imposed. As we will see in Remark~\ref{rem-nondeg2}, this stronger assumption is not satisfied in the applications which we are going to develop in Section~\ref{sec2.2}.
\hfill$\lhd$
\end{remark}

We are now in a position to state the result about bifurcation of periodic solutions (with prescribed energy) from a non-degenerate periodic manifold of system \eqref{hs}; it corresponds to \cite[Theorem~1.4]{We-73} in the case of the action of the group $\mathbb{Z}_{n}$ for $n=1$.

\begin{theorem}\label{teo-weinstein0}
Let $(M,\omega)$ be a symplectic manifold such that the form $\omega$ is exact, let $\mathcal{K}_{0} \in \mathcal{C}^{\infty}(M)$ and let
$\Sigma \subset \mathrm{Per}_{\mathcal{K}_{0}}^H$ a non-degenerate periodic manifold, for some $H \in \mathbb{R}$.
Finally, consider a function $\mathcal{H} \in \mathcal{C}^{\infty}(M \times \mathbb{R})$ such that $\mathcal{H}_{0} = \mathcal{K}_{0}$, where
$\mathcal{H}_\varepsilon = \mathcal{H}(\cdot,\varepsilon)$. 
Then, denoting by $m$ the least integer greater or equal to $\mathrm{Cat}(\Sigma)/2$, for every neighborhood $\mathcal{U} \subset M \times (0,+\infty)$ of $\Sigma$, there exists $\varepsilon^{*} > 0$ such that,
if $\vert \varepsilon \vert < \varepsilon^{*}$, there exist $\tau_{1},\ldots,\tau_m > 0$ and there exist 
$z_{1},\ldots,z_m$ distinct periodic solutions of the Hamiltonian system
\begin{equation*}
\dot z = X_{\mathcal{H}_\varepsilon}(z)
\end{equation*}
such that, for every $i=1,\ldots,m$:
\begin{itemize}
\item[$(I)$] $\mathcal{H}_\varepsilon(z_i(t)) = H$, for every $t \in \mathbb{R}$,
\item[$(II)$] $\tau_i$ is a period for $z_i$,
\item[$(III)$] $\{(z_i(t),\tau_i) \colon t \in \mathbb{R} \} \subset \mathcal{U}$.
\end{itemize}
\end{theorem}

Notice that $(III)$ provides a localization information for both the orbit $\{z_i(t)\}_{t \in \mathbb{R}}$ on $M$ and the period $\tau_i$, which can thus be chosen arbitrarily near orbits and periods of the solutions of the unperturbed problem. We also point out that, actually, the exactness condition on the symplectic form $\omega$ can be removed, as shown later in \cite{We-78}.

\subsection{Bifurcation from periodic invariant tori}\label{sec2.2}

Let $\mathcal{D} \subset \mathbb{R}^n$ ($n \geq 2$) be an open set and let $\mathcal{K}_{0}\colon \mathcal{D} \to \mathbb{R}$ and $\mathcal{R}\colon \mathcal{D} \times \mathbb{R}^n \to \mathbb{R}$ be
$\mathcal{C}^{\infty}$ functions; moreover, we assume that $\mathcal{R}$ is $2\pi$-periodic in each variable $\varphi_i$, where $\varphi = (\varphi_{1},\ldots,\varphi_{n})$. 
Defining the symplectic form
\begin{equation}\label{def-omega}
\omega = \mathrm{d}\varphi \wedge \mathrm{d}I
\end{equation}
and, for every $\varepsilon \in \mathbb{R}$, the function $\mathcal{H}_\varepsilon \colon \mathcal{D} \times \mathbb{R}^n \to \mathbb{R}$ as
\begin{equation*}
\mathcal{H}_\varepsilon (I,\varphi)=\mathcal{K}_{0}(I)+\varepsilon \, \mathcal{R}(I,\varphi),
\end{equation*}
the Hamiltonian system associated with the Hamiltonian $\mathcal{H}_\varepsilon$ with respect to the symplectic structure induced by the form $\omega$ reads as
\begin{equation} \label{eq-quasiint}
\begin{cases}
\, \dot I = - \varepsilon \, \nabla_\varphi \mathcal{R}(I,\varphi), \\
\, \dot \varphi = \nabla \mathcal{K}_{0}(I) + \varepsilon \, \nabla_I \mathcal{R}(I,\varphi).
\end{cases}
\end{equation}
For $\varepsilon = 0$, system \eqref{eq-quasiint} reduces
to the completely integrable system in action-angle coordinates
\begin{equation} \label{eq-int}
\begin{cases}
\, \dot I = 0, \\
\, \dot \varphi = \nabla \mathcal{K}_{0}(I), 
\end{cases}
\end{equation}
whose solutions are given by
\begin{equation} \label{eq-solintegrabile}
I(t)=I^{*},\quad \varphi(t)=\varphi^{*} +t\, \nabla \mathcal{K}_{0}(I^{*}),\quad t\in \mathbb{R},
\end{equation}
being $(I^{*},\varphi^{*})\in \mathcal{D}\times \mathbb{R}^n$ the initial datum.

In the following, \eqref{eq-quasiint} is meant as the the lifting to the covering space of the corresponding Hamiltonian system
on the manifold
\begin{equation}\label{def-M}
M = \mathcal{D} \times \mathbb{T}^n,
\end{equation}
where $\mathbb{T}^n = \mathbb{R}^n / 2\pi \, \mathbb{Z}^n$ is the $n$-dimensional torus.
Accordingly, we say that a solution of
\eqref{eq-quasiint} is periodic if its projection on $M$ is periodic, that is, if there exists $\tau^{*} > 0$ such that
\begin{equation*}
I(\tau^{*}) = I(0), \qquad \varphi(\tau^{*}) - \varphi(0) \in 2\pi \, \mathbb{Z}^n.
\end{equation*}
Hence, if $I^{*} \in \mathcal{D}$ and $\tau^{*}> 0$ are such that
\begin{equation}\label{periodic-filled}
\tau^{*} \nabla \mathcal{K}_{0}(I^{*}) \in 2\pi \, \mathbb{Z}^n \setminus \{0\},
\end{equation}
and if $\tau^{*}$ is the least positive number with the above property, 
then the torus 
\begin{equation}\label{def-torus}
\mathcal{T}_{I^{*}} = \{I^{*}\} \times \mathbb{T}^n
\end{equation}
is filled by non-constant periodic solutions of system \eqref{eq-int} of minimal period $\tau^{*}$.

In this setting, we are interested in the existence of periodic solutions to \eqref{eq-quasiint}, having prescribed energy,
bifurcating from the torus $\mathcal{T}_{I^{*}}$. Our result is the following.

\begin{theorem} \label{teo-weinstein}
Let us assume that condition \eqref{periodic-filled} holds true and that 
\begin{equation}\label{KAM-iso}
\mathrm{det} \,
\begin{pmatrix}
\nabla^2 \mathcal{K}_{0}(I^{*}) & \quad \nabla \mathcal{K}_{0}(I^{*})^{\intercal} \vspace{3pt} \\
\nabla \mathcal{K}_{0}(I^{*}) & \quad 0
\end{pmatrix} 
\neq 0,
\end{equation}
where $\nabla \mathcal{K}_{0}(I^{*})$ is meant as a row vector.
Let $H^{*} = \mathcal{K}_{0}(I^{*})$ and let $m$ be the least integer greater or equal to $(n+1)/2$.
Then, for every $\varsigma >0$, there exists $\varepsilon^{*} > 0$ such that, if $\vert \varepsilon\vert \leq \varepsilon^{*}$,
there exist $\tau_{1},\ldots,\tau_m > 0$ and there exist $(I_{1},\varphi_{1}), \ldots, (I_n,\varphi_n)$ distinct periodic solutions of the Hamiltonian system \eqref{eq-quasiint}
such that, for every  $i=1,\ldots,m$:
\begin{itemize}
\item[$(I)$] $\mathcal{K}_\varepsilon (I_i(t),\varphi_i(t))=H^{*}$, for every $t \in \mathbb{R}$,
\item[$(II)$] $\tau_i$ is the minimal period for $(I_i,\varphi_i)$,
\item[$(III)$] the conditions
\begin{equation} \label{eq-distanzaperiodi}
\vert \tau_i-\tau^{*}\vert <\varsigma
\end{equation}
and 
\begin{equation} \label{eq-distanzaangoli}
\vert I_i(t)-I^{*}\vert <\varsigma,\quad \vert \varphi_i (t)-\varphi_i (0)-t\, \nabla \mathcal{K}_{0}(I^{*})\vert <C_*\varsigma,\quad \text{$\forall \, t\in \mathopen{[}0,\tau_i \mathclose{]}$,}
\end{equation}
are satisfied, where $C_* > 0$ is a constant depending only on $\mathcal{K}_{0}$ and $\mathcal{R}$.
\end{itemize}
\end{theorem}

\begin{proof}
Regarding \eqref{eq-quasiint} as a Hamiltonian system on the symplectic manifold $(M,\omega)$, with $M$ and $\omega$ 
as in \eqref{def-M} and \eqref{def-omega} respectively, and noticing that $\omega$ is exact on $M$, we are going to apply Theorem~\ref{teo-weinstein0} with the choice
\begin{equation*}
\Sigma = \mathcal{T}_{I^{*}} \times \{\tau^{*}\} \subset \mathrm{Per}_{\mathcal{K}_{0}}^{H^{*}},
\end{equation*}
with $\mathcal{T}_{I^{*}}$ as in \eqref{def-torus}. 
Of course, $\Sigma$ is diffeomorphic to a $n$-dimensional torus, so that condition $(i)$ holds true; moreover, since the period is constantly equal to $\tau^{*}$ for every solution on $\mathcal{T}_{I^{*}}$, condition $(ii)$ is trivially satisfied. 

It thus remains to the check that the non-degeneracy condition $(iii)$ is satisfied.
So, let us fix $(\xi,\tau) \in \Sigma$, that is $\tau = \tau^{*}$ and $\xi = (I^{*},\varphi)$ for some $\varphi \in [0,2\pi)$.
Then,
\begin{equation*}
T_{\xi}(\mathcal{K}_{0}^{-1}(H^{*})) = 
\bigl{\{} \eta = (v,\phi) \in \mathbb{R}^n \times \mathbb{R}^n  \colon \langle \nabla \mathcal{K}_{0}(I^{*}), v \rangle = 0 \bigr{\}}
\end{equation*}
and, of course,
\begin{equation*}
T_{\xi} \pi(\Sigma) = \{0\} \times \mathbb{R}^n.
\end{equation*}
On the other hand, 
\begin{equation*}
X_{\mathcal{K}_{0}}(\xi) = (0,\nabla \mathcal{K}_{0}(I^{*})),
\end{equation*}
while recalling \eqref{eq-solintegrabile} we easily infer that the monodromy operator $P$
is given by
\begin{equation}\label{formula-P}
P \eta = 
\begin{pmatrix}
\mathrm{Id}_{\mathbb{R}^{n}} & \quad 0 \vspace{3pt} \\
\tau^{*} \nabla^2 \mathcal{K}_{0}(I^{*}) & \quad \mathrm{Id}_{\mathbb{R}^n} 
\end{pmatrix} \eta
= (v, \tau^{*} \nabla^2 \mathcal{K}_{0}(I^{*}) v + \phi),
\end{equation}
for every $\eta = (v,\phi) \in T_{\xi}(\mathcal{K}_{0}^{-1}(H^{*}))$.
Hence
\begin{equation*}
(\mathrm{Id}_{\mathbb{R}^{2n}}-P) \eta = (0,\tau^{*} \nabla^2 \mathcal{K}_{0}(I^{*}) v)
\end{equation*} 
and thus
condition $(iii)$ is satisfied if and only if, given $\lambda \in \mathbb{R}$,
the only vector $v \in \mathbb{R}^n$ satisfying 
\begin{equation*}
\tau^{*} \nabla^2 \mathcal{K}_{0}(I^{*}) v = \lambda \nabla\mathcal{K}_{0}(I^{*})
\end{equation*}
and such that $\langle \nabla \mathcal{K}_{0}(I^{*}), v \rangle = 0$ is the vector $v = 0$.

Let us assume by contradiction that this is not case. Hence, setting $\lambda' = -\lambda/\tau^{*}$ we have found a non trivial solution
$(v,\lambda')$ of the linear homogeneous system in $\mathbb{R}^{n+1}$ whose coefficient matrix is
\begin{equation*}
\begin{pmatrix}
\nabla^2 \mathcal{K}_{0}(I^{*}) & \quad \nabla \mathcal{K}_{0}(I^{*})^{\intercal} \vspace{3pt} \\
\nabla \mathcal{K}_{0}(I^{*}) & \quad 0
\end{pmatrix}, 
\end{equation*}
thus contradicting assumption \eqref{KAM-iso}.

Summing up, all the assumptions of Theorem~\ref{teo-weinstein0} are satisfied.
Hence, denoting by $m$ the the least integer greater or equal to $\mathrm{Cat}(\Sigma)/2 = \mathrm{Cat}(\mathbb{T}^n)/2 = (n+1)/2$, 
given any $\varsigma >0$, we can consider the open neighborhood $U$ of $\Sigma$ given by
\begin{equation*}
U = B_\varsigma(I^{*}) \times \mathbb{T}^n \times (\tau^{*}- \varsigma, \tau^{*} +\varsigma)
\end{equation*}
and apply Theorem~\ref{teo-weinstein0} to find, for $\vert \varepsilon \vert \leq \varepsilon^{*}$, $m$ numbers $\tau_i >0$ satisfying \eqref{eq-distanzaperiodi} and $m$ solutions $(I_i,\varphi_i)$ of system \eqref{eq-quasiint}, with period $\tau_i$, energy $H$ and satisfying the first condition in \eqref{eq-distanzaangoli}. We now claim that condition \eqref{eq-distanzaangoli} is satisfied, as well. To prove this, 
we assume without loss of generality that $\varepsilon^{*} \leq \varsigma$ and we first notice that the first condition in \eqref{eq-distanzaangoli}, together with the regularity of $\mathcal{K}_{0}$ and $\mathcal{R}$, implies that
\begin{equation*}
\vert \nabla \mathcal{K}_{0}(I_i(t)) -  \nabla \mathcal{K}_{0}(I^{*}) \vert < \frac{C _*}{2 \tau_i} \varsigma, 
\quad \text{for every $t\in [0,\tau_i]$,}
\end{equation*}
and that
\begin{equation*}
\vert \nabla_I \mathcal{R}(I_i(t),\varphi_i(t )) \vert < \frac{C _*}{2\tau_i},
\quad \text{for every $t\in [0,\tau_i]$,}
\end{equation*}
where $C_* > 0$ is a constant depending only on $\mathcal{K}_{0}$ and $\mathcal{R}$.
Since
\begin{align*}
& \varphi_i(t) - \varphi_i(0) - t \,\nabla \mathcal{K}_{0}(I^{*})  = 
\\
&= \int_{0}^t \bigl{(} \dot\varphi_i(s) - \nabla \mathcal{K}_{0}(I^{*})\bigr{)} \,\mathrm{d}s  
\\
& = \int_{0}^t \bigl{(} \nabla \mathcal{K}_{0}(I_i(s)) -  \nabla \mathcal{K}_{0}(I^{*}) \bigr{)} \,\mathrm{d}s 
+ \varepsilon \int_{0}^t  \nabla_I \mathcal{R}(I_i(s),\varphi_i(s)) \,\mathrm{d}s,
\end{align*}
we thus have
\begin{equation*}
\left \vert \varphi_i(t) - \varphi_i(0) - t \,\nabla \mathcal{K}_{0}(I^{*})  \right \vert \leq C_* \varsigma,
\quad \text{for every $t \in [0,\tau_i]$,}
\end{equation*}
as desired.

It thus remains to show that $\tau_i$ is the minimal period of $(I_i,\varphi_i)$. 
To this end, we first observe that \eqref{eq-distanzaperiodi} and \eqref{eq-distanzaangoli} imply, for $\varsigma$ small enough, that
\begin{equation*}
\varphi_i(\tau_i) - \varphi_i(0) = \tau^{*} \nabla\mathcal{K}_{0}(I^{*}).
\end{equation*}
Now, let us assume by contradiction that the period $\tau_i$ is not minimal, that is, $\tau_i/\ell$ is also a period for some integer $\ell \geq 2$.
Hence, both the components of the vector $\varphi_i(\tau_i) - \varphi_i(0)$ must be a multiple of $\ell$, and thus 
\begin{equation*}
\frac{\tau^{*}}{\ell}\nabla\mathcal{K}_{0}(I^{*}) \in 2\pi \, \mathbb{Z}^n \setminus \{0\}.
\end{equation*}
This contradicts that fact that $\tau^{*}$ is the least number satisfying \eqref{periodic-filled} (that is to say, $\tau^{*}$ is not the minimal period of the solution on the unperturbed torus $\mathcal{T}_{I^{*}}$).
\end{proof}

\begin{remark} \label{rem-numerosol} 
In Section~\ref{sec-autonomo} we will apply Theorem~\ref{teo-weinstein} in the case of two degrees of freedom, that is, $n=2$.
It is easy to see that in this situation the non-degeneracy condition \eqref{KAM-iso} writes equivalently as
\begin{equation} \label{eq-kamiso}
\langle \nabla^2 \mathcal{K}_{0}(I^{*}) (\partial_{2} \mathcal{K}_{0}(I^{*}), -\partial_{1} \mathcal{K}_{0}(I^{*})),(\partial_{2} \mathcal{K}_{0}(I^{*}), -\partial_{1} \mathcal{K}_{0}(I^{*})) \rangle\neq 0.
\end{equation}
We also explicitly observe that the number of bifurcating periodic solutions of \eqref{eq-quasiint} is $2$.
\hfill$\lhd$
\end{remark}

\begin{remark}\label{rem-nondeg2}
From \eqref{formula-P}, it is immediate to deduce that $1$ is the only eigenvalue of $P$: thus, its algebraic multiplicity is $2n-1$, a number which  is (for $n \geq 2$) strictly greater than $n$, that is the dimension of the manifold $\Sigma$. Therefore, the non-degeneracy condition of \cite{We-77} is not satisfied in our setting, cf.~Remark~\ref{rem-nondeg}.
\hfill$\lhd$
\end{remark}

\section{Main results} \label{sec-autonomo}

In this section, we apply Theorem~\ref{teo-weinstein} to study bifurcation of periodic solutions for the two relativistic problems described in the Introduction: the Kepler problem with Levi-Civita correction \eqref{eq-pert2} and the Kepler problem with relativistic differential operator \eqref{eq-pert1}.

The two problems share some structural features which enable us to introduce a common strategy to deal with. Indeed, they both have a Hamiltonian formulation and the unperturbed problems are integrable, due to the presence of two independent first integrals in involution (the energy and the angular momentum).

The construction of action-angle coordinates in both the cases is based on a phase-plane analysis for the radial component of the solutions in polar coordinates.

For the purpose of applying Theorem~\ref{teo-weinstein}, we first characterize the pairs energy/angular momentum for which periodic solutions exist and then we directly prove the validity of the non-degeneracy condition \eqref{eq-kamiso}.

\subsection{The Kepler problem with Levi-Civita correction} \label{subsec-levicivita}

In this section we deal with the problem
\begin{equation} \label{eq-levicivitacompl}
m \ddot x = -\kappa \dfrac{x}{\vert x\vert ^{3}} - 2\lambda \dfrac{x}{\vert x\vert ^4} + \varepsilon \,\nabla U(x),
\end{equation}
where $m>0$, $\kappa>0$, $\lambda>0$, $U \in \mathcal{C}^{\infty}(\mathbb{R}^2,\mathbb{R})$, and $\varepsilon \in \mathbb{R}$.

Our result is the following (by now, we do not give any physical interpretation for the constants $\kappa$ and $\lambda$: this aspect will be treated in the discussion leading to Corollary~\ref{cor-levi}).

\begin{theorem} \label{teo-levicivita}
Let $H^{*}<0$ and $N\in \mathbb{N}$. Then, for every $\varsigma>0$ there exists $\varepsilon^{*}=\varepsilon^{*}(N,\varsigma)>0$ such that for every $\varepsilon\in \mathbb{R}$, with $\vert \varepsilon\vert \leq \varepsilon^{*}$, equation \eqref{eq-levicivitacompl} has at least $2N$ periodic solutions $x^i_{1},\ldots, x^i_N$, $i=1,2$,	of energy $H^{*}$.

Moreover, for every $j=1,\ldots, N$ and $i=1,2$, the solution $x^i_j$ has minimal period $T_{i,j}$ satisfying
\begin{equation*}
\left\vert T_{i,j}-\dfrac{\pi \kappa \sqrt{m}}{\sqrt{2} (-H^{*})^{3/2}}\right\vert <\varsigma
\end{equation*}
and winding number in its minimal period
\begin{equation*}
\left\lfloor\dfrac{\sqrt{-2H^{*}}}{\kappa}\, \sqrt{2\lambda +\dfrac{\kappa^2}{-2H^{*}}}\right\rfloor+j,
\end{equation*}
where $\lfloor\cdot\rfloor$ denotes the integer part.
\end{theorem}

Theorem~\ref{teo-levicivita} follows from the application of Theorem~\ref{teo-weinstein} to the Hamiltonian system associated to \eqref{eq-levicivitacompl}. Indeed, let us first recall that \eqref{eq-levicivitacompl} can be written as a Hamiltonian system, with respect to the variables $(x,p)=(x,m\dot{x})$. The Hamiltonian is given by
\begin{equation*}
\mathcal{H}_\varepsilon(x,p)=\dfrac{1}{2m} \vert p\vert ^2-\dfrac{\kappa}{\vert x\vert }-\dfrac{\lambda}{\vert x\vert ^2}-\varepsilon \, U(x),
\end{equation*}
whose values correspond to values of the energy 
\begin{equation} \label{eq-energiaLC}
E_\varepsilon (x,\dot{x})=\dfrac{1}{2} m \vert \dot{x}\vert ^2 -\dfrac{\kappa}{\vert x\vert }-\dfrac{\lambda}{\vert x\vert ^2}-\varepsilon \, U(x).
\end{equation}

Now, let us observe that the Hamiltonian system associated to $\mathcal{H}_\varepsilon$ is nearly integrable and that can be transformed in the form \eqref{eq-quasiint} passing to action-angle coordinates. The construction is quite standard (cf.~\cite[Section~2]{BoDaFe-pp}) and it relies on the fact that the unperturbed Hamiltonian system associated to $\mathcal{H}_{0}$ has a second first integral independent from $\mathcal{H}_{0}$ and in involution with it, which is the angular momentum defined by
\begin{equation} \label{eq-momanglevi}
\mathcal{L}_{0}(x,p) = \langle x, Jp \rangle, \qquad \text{where } J = \begin{pmatrix} 0 & 1 \\
-1 & 0\end{pmatrix}.
\end{equation}
In order to define the change of variable, we focus on the unperturbed problem 
\begin{equation} \label{eq-levicivita}
m\ddot x = -\kappa \dfrac{x}{\vert x\vert ^{3}} - 2\lambda \dfrac{x}{\vert x\vert ^4}
\end{equation}
and we pass to polar coordinates $x=r e^{i\vartheta}$. 

Let us define
\begin{equation*}
\Lambda =\biggl{\{}(H,L)\in\mathbb{R}^{2} \colon L^2\in \left(\dfrac{2\lambda}{m},+\infty\right), \; H\in \left(-\dfrac{\kappa^{2}}{2(mL^{2}-2\lambda)}, 0\right) \biggr{\}}.
\end{equation*}
Following \cite{BoDaFe-pp}, we know that the radial component $r$ of the solutions of \eqref{eq-levicivita} satisfies
\begin{equation*}
\dfrac{1}{2} m \dot{r}^2+\dfrac{mL^2}{2r^2}-\dfrac{\kappa}{r}-\dfrac{\lambda}{r^2}=H.
\end{equation*}
From the above relation we infer that for every $L\in \mathbb{R}$ the radial component moves accordingly to a nonlinear oscillator driven by the effective potential
\begin{equation*}
W(r;L)=\dfrac{mL^2-2\lambda}{2r^2}-\dfrac{\kappa}{r}.
\end{equation*}
When $mL^2>2\lambda$, this potential satisfies
\begin{equation*}
\lim_{r\to 0^+} W(r;L)=+\infty, \quad \lim_{r\to +\infty} W(r;L)=0,
\end{equation*}
and it is strictly decreasing in $(0,(mL^2-2\lambda)/\kappa)$ and strictly increasing in $((mL^2-2\lambda)/\kappa,+\infty)$. The global minimum of $W(\cdot;L)$ is
\begin{equation} \label{eq-minLC}
w_{\min} (L)=-\dfrac{\kappa^2}{2(mL^2-2\lambda)}.
\end{equation}
As a consequence, for every $(H,L)\in \Lambda$ the orbits in the $(r,\dot{r})$-plane are closed curves.

In \cite[Section~4.1]{BoDaFe-pp} it is proved that the corresponding radial period is given by
\begin{equation}\label{eq-perradLC}
T(H,L)=\dfrac{\pi \kappa \sqrt{m}}{\sqrt{2} (-H)^{3/2}}
\end{equation}
and that the angular displacement in this period (the so-called \textit{apsidal angle}) is 
\begin{equation*}
\Theta (H,L)=\dfrac{2\pi  \sqrt{m} L}{\sqrt{mL^2-2\lambda}}.
\end{equation*}
For every $(H,L)\in \Lambda$, the set
\begin{equation*}
\mathcal{T}_{(H,L)} = \bigl{\{}(x,p)\in(\mathbb{R}^{2}\setminus\{0\})\times\mathbb{R}^{2} 
\colon \mathcal{H}_{0}(x,p)=H, \; \mathcal{L}_{0}(x,p)=L \bigr{\}}
\end{equation*}
is then diffeomorphic to a two-dimensional torus $\mathbb{T}^{2}$ (cf.~\cite{Ar-89, BeFa-notes}). 

Action-angle coordinates can be constructed in a neighborhood of $\mathcal{T}_{(H^{*},L^{*})}$, for every fixed $(H^{*},L^{*})\in\Lambda$. Indeed, given $(H^{*},L^{*})\in\Lambda$, for every $(H,L)$ in a neighborhood $\Upsilon^{*}$ of $(H^{*},L^{*})$, let $\mathcal{A}(H,L)$ be the area of the bounded region enclosed by the orbit corrsponding to $(H,L)$ in the $(r,\dot r)$-plane and define $\Pi \colon \Upsilon^{*}\to \mathbb{R}^4$ by
\begin{equation} \label{eq-defaaLC}
\Pi(H,L)=\left(\dfrac{1}{2\pi}\mathcal{A}(H,L)+L, L, 2\pi \dfrac{\mu}{T(H,L)},(\Theta(H,L)-2\pi)\dfrac{\mu}{T(H,L)}+\psi\right),
\end{equation}
for every $(H,L)\in \Upsilon^{*}$, where $\mu \in [0,T(H,L))$ is the time needed to reach the point $(r,\dot{r})$ in the $(r,\dot{r})$-plane along the orbit starting form the pericenter and $\psi\in [0,2\pi)$ is the angle, measured in the counter-clockwise sense from $\vartheta=0$, of the greatest non-positive instant in which $x$ lies at the pericenter.

The map
\begin{equation*}
\Xi \colon (x,p) \mapsto \Pi(\mathcal{H}_{0}(x,p),\mathcal{L}_{0}(x,p))
\end{equation*}
is well-defined for $(x,p)$ in a neighbourhood of the fixed torus $\mathcal{T}_{(H^{*},L^{*})}$ and it is well-known that it provides a symplectic diffeomorphism from this set onto its image.

In the new variables, the Hamiltonian system is of the form \eqref{eq-quasiint} with 
\begin{equation}\label{eq-def-kLV}
\mathcal{K}_{0}(I_{1},I_{2}) = \mathcal{H}_{0}(\Xi^{-1}(I_{1},I_{2},\varphi_{1},\varphi_{2})).
\end{equation}
We are now in a position to prove Theorem~\ref{teo-levicivita}.

\begin{proof}[Proof of Theorem~\ref{teo-levicivita}] From the previous discussion, we know that in the action-angle variables $(I_{1},I_{2},\varphi_{1},\varphi_{2})$ the Hamiltonian system corresponding to \eqref{eq-levicivitacompl} is a smooth perturbation of the Hamiltonian system associated to the Hamiltonian $\mathcal{K}_{0}$ defined in \eqref{eq-def-kLV}.

Now, let us fix $H^{*}<0$, $N\in \mathbb{N}$, $j\in \{1,\ldots N\}$ and let us set 
\begin{equation*}
k^{*}_j=\left \lfloor \dfrac{\sqrt{-2H^{*}}}{\kappa}\, \sqrt{2\lambda +\dfrac{\kappa^2}{-2H^{*}}}\right\rfloor+j.
\end{equation*} 
Let
\begin{equation*}
L^{*}_{\max}=\dfrac{1}{\sqrt{m}}\, \sqrt{2\lambda+\dfrac{\kappa^2}{-2H^{*}}}.
\end{equation*}
Notice that $L^{*}_{\max}$ is the unique solution of
\begin{equation*}
H^{*}=-\dfrac{\kappa^2}{2(mL^2-2\lambda)}
\end{equation*}
and that $(H^{*},L^{*})\in \Lambda$, for every $L^{*}\in (\sqrt{2\lambda/m},L^{*}_{\max}) $. Moreover, since
\begin{equation*}
2\pi \dfrac{\sqrt{-2H^{*}}}{\kappa}\, \sqrt{2\lambda +\dfrac{\kappa^2}{-2H^{*}}}=\Theta (H^{*},L^{*}_{\max}),
\end{equation*}
from the the monotonicity of $\Theta (H^{*},\cdot)$ and the choice of $k^{*}_j$ we deduce that there exists a unique $L^{*}_j\in (\sqrt{2\lambda/m},L^{*}_{\max}) $ such that 
\begin{equation} \label{eq-numerogiriLC}
\Theta (H^{*},L^{*}_j)= 2\pi\, k^{*}_j.
\end{equation}
Since $(H^{*},L^{*}_j)\in \Lambda$, as already observed, we deduce that the solution of the unperturbed equation \eqref{eq-levicivita} having energy $H^{*}$ and angular momentum $L^{*}_j$ is periodic of period $T(H^{*},L^{*}_j)$, with $T$ as in \eqref{eq-perradLC}, and it has winding number $2\pi k^{*}_j$ in a period.

As far as \eqref{eq-kamiso} is concerned, with $\mathcal{K}_{0}$ given in \eqref{eq-def-kLV}, we prove its validity for every $(I_{1},I_{2})$ in the domain of $\mathcal{K}_{0}$, which will be denoted by $\mathcal{O}^{*}$. Indeed, let us first observe that by construction of the action-angles variables we have $\mathcal{K}_{0} (I_{1},I_{2})=H(I_{1},I_{2})$, where $(H,L)=(H(I_{1},I_{2}),L(I_{1},I_{2}))$ satisfies
\begin{equation*}
I_{1}=\dfrac{1}{2\pi}\, \mathcal{A} (H,L)+L,
\qquad I_{2}=L
\end{equation*}
(recall \eqref{eq-defaaLC}). From these relations we immediately deduce that $L(I_{1},I_{2})=I_{2}$ and that $H=H(I_{1},I_{2})$ fulfills
\begin{equation} \label{eq-invertireLC}
\dfrac{1}{2\pi}\, \mathcal{A} (H,I_{2})=I_{1}-I_{2}.
\end{equation}
Now, let us recall that $\mathcal{A}(H,L)=\partial_H T(H,L)$, for every $(H,L)\in \Lambda$ (see \cite[p. 282]{Ar-89}). Taking into account \eqref{eq-perradLC}, we infer that 
\begin{equation} \label{eq-pr201}
\mathcal{A} (H,L)=\dfrac{2\pi \kappa}{\sqrt{-2H}}+\zeta(L),
\quad \text{for every $(H,L)\in \Lambda$,}
\end{equation}
for some differentiable function $\zeta \colon (-\infty,-\sqrt{2\lambda/m})\cup (\sqrt{2\lambda/m},+\infty)\to \mathbb{R}$. For every $L\in (-\infty,-\sqrt{2\lambda/m})\cup (\sqrt{2\lambda/m},+\infty)$, the value $\zeta(L)$ can be easily obtained recalling that for the minimum admissible value of the energy $w_{\min}(L)$ (cf.~\eqref{eq-minLC}) we have
\begin{equation*}
\lim_{H\to w_{\min}(L)} \mathcal{A} (H,L)=0.
\end{equation*}
From \eqref{eq-pr201} we then deduce that
\begin{equation*}
\zeta (L)=- \dfrac{2\pi \kappa}{\sqrt{-2w_{\min}(L)}}=-2\pi \sqrt{m}\sqrt{mL^2-2\lambda},
\end{equation*}
thus concluding that
\begin{equation*}
\mathcal{A} (H,L)=2\pi  \sqrt{m} \, \left(\dfrac{\kappa}{\sqrt{-2H}}- \sqrt{mL^2-2\lambda}\right),
\quad \text{for every $(H,L)\in \Lambda$.}
\end{equation*}
As a consequence, \eqref{eq-invertireLC} reduces to
\begin{equation*}
\dfrac{\kappa  \sqrt{m}}{\sqrt{-2H}}-  \sqrt{m}\, \sqrt{mI_{2}^2-2\lambda}=I_{1}-I_{2}.
\end{equation*}
By means of a standard computations we obtain 
\begin{align*}
\mathcal{K}_{0}(I_{1},I_{2})=H(I_{1},I_{2})
&=-\dfrac{m\kappa^2}{2}\, \left(I_{1}-I_{2}+ \sqrt{m} \, \sqrt{mI_{2}^2-2\lambda}\right)^{\!-2}
\\
& = -\dfrac{m\kappa^2}{2}\, \gamma(I_{1},I_{2})^{-2},
\quad \text{for every $(I_{1},I_{2})\in \mathcal{O}^{*}$,}
\end{align*}
where $\gamma (I_{1},I_{2})=I_{1}-I_{2}+ \sqrt{m}\, \sqrt{m I_{2}^2-2\lambda}$. We then deduce that
\begin{align*}
\partial_{1} \mathcal{K}_{0}(I_{1},I_{2}) &= m\kappa^2 \gamma(I_{1},I_{2})^{-3}
\\
\partial_{2} \mathcal{K}_{0}(I_{1},I_{2}) &= m \kappa^2 \gamma(I_{1},I_{2})^{-3}\, \left(m\sqrt{m}I_{2}\, \bigl{(}mI_{2}^2-2\lambda\bigr{)}^{\!-\frac{1}{2}} -1\right)
\end{align*}
and
\begin{align*}
\partial^2_{11} \mathcal{K}_{0}(I_{1},I_{2})
&= -3m\kappa^2 \gamma(I_{1},I_{2})^{-4}
\\
\partial^2_{12} \mathcal{K}_{0}(I_{1},I_{2}) 
&= -3m\kappa^2 \gamma(I_{1},I_{2})^{-4}\, \left(m\sqrt{m} I_{2}\, \bigl{(}m I_{2}^2-2\lambda\bigr{)}^{\!-\frac{1}{2}} -1\right)
\\
\partial^2_{22} \mathcal{K}_{0}(I_{1},I_{2})
&= -3m \kappa^2 \gamma(I_{1},I_{2})^{-4}\, \left(m\sqrt{m}I_{2}\, \bigl{(}mI_{2}^2-2\lambda\bigr{)}^{\!-\frac{1}{2}} -1\right)^2
\\
&-2\lambda m^2\sqrt{m}\kappa^2 \gamma(I_{1},I_{2})^{-3}\, \bigl{(}mI_{2}^2-2\lambda\bigr{)}^{\!-\frac{3}{2}},
\end{align*}
for every $(I_{1},I_{2})\in \mathcal{O}^{*}$. Hence, we obtain
\begin{equation*}
\det \mathcal{K}_{0} (I_{1},I_{2})=6\lambda m^3\sqrt{m}\kappa^4 \gamma(I_{1},I_{2})^{-7}\, \bigl{(}mI_{2}^2-2\lambda\bigr{)}^{\!-\frac{3}{2}}>0,
\end{equation*}
for every $(I_{1},I_{2})\in \mathcal{O}^{*}$. On the other hand, since
\begin{equation*}
\partial^2_{11} \nabla^2 \mathcal{K}_{0}(I_{1},I_{2})=-3m\kappa^2 \gamma(I_{1},I_{2})^{-4}<0,
\quad \text{for every $(I_{1},I_{2})\in \mathcal{O}^{*}$,}
\end{equation*}
we deduce that the trace of $\nabla^2 \mathcal{K}_{0} (I_{1},I_{2})$ is negative, for every $(I_{1},I_{2})\in \mathcal{O}^{*}$. Hence, $\nabla^2 \mathcal{K}_{0} (I_{1},I_{2})$ is definite negative, for every $(I_{1},I_{2})\in \mathcal{O}^{*}$. Condition \eqref{eq-kamiso} is then satisfied for every $I=(I_{1},I_{2})\in \mathcal{O}^{*}$.

We can now conclude. Indeed, let  $(I^{*}_j,\varphi^{*}_j)=\Pi(H^{*},L^{*}_j)$, where $\Pi$ is defined in \eqref{eq-defaaLC}.
From Theorem~\ref{teo-weinstein} we deduce that for every $\varsigma>0$ there exists $\varepsilon^{*}_{j}$ such that for every $\varepsilon\in \mathbb{R}$ with $\vert \varepsilon\vert \leq \varepsilon^{*}_{j}$ equation \eqref{eq-levicivitacompl} has two periodic solutions $x^i_j$, $i=1, 2$, with energy $H^{*}$ and satisfying the required property of the minimal period. Moreover, the condition on the winding number of $x^i_j$ follows from \eqref{eq-numerogiriLC} taking into account \eqref{eq-distanzaangoli} and the continuity of the winding number (cf.~the argument in the proof of \cite[Theorem~3.2]{BoDaFe-21}).
The result then follows taking $\varepsilon^{*}=\min \{\varepsilon^{*}_{j} \colon j=1,\ldots , N\}$.
\end{proof}

\begin{remark}\label{rem-3.1}
We mention that the isoenergetic non-degeneracy of equation \eqref{eq-levicivita} was already proved in \cite{NuCaLl-91} using a different set of action-angles coordinates. Our result can thus be seen as a periodic counterpart of \cite{NuCaLl-91} which deals with the persistence of quasi-periodic invariant tori.
\hfill$\lhd$
\end{remark}

We end this section with a corollary of Theorem~\ref{teo-levicivita}, concerning the original relativistic problem introduced by Levi-Civita in \cite[Sect.~11 of Ch.~1, Sect.~9 of Ch.~2]{LC-28}.

In the context of the relativistic approach by Levi-Civita, the parameters $\kappa$ and $\lambda$ are related to the speed of light $c$ and the einsteinian energy $E<0$ by
\begin{equation} \label{eq-kappalambda}
\kappa (E,c)=mGM + \dfrac{4EGM}{c^2},\quad 2\lambda (c) = \dfrac{6G^2 M^2}{c^2},
\end{equation}
where $G$ is gravitational constant and $M$ is the mass of the black hole. Moreover, physically meaningful solutions have approximate energy (in the sense of \eqref{eq-energiaLC}) still equal to $E$. As a consequence, we are led to the study of the problem
\begin{equation} \label{eq-LCfinale}
\begin{cases}
\, m \ddot x = -\left(mGM + \dfrac{4EGM}{c^2}\right) \dfrac{x}{\vert x\vert ^{3}} - \dfrac{6G^2 M^2}{c^2} \dfrac{x}{\vert x\vert ^4} + \varepsilon \,\nabla U(x),
\vspace{5pt}\\
\, \dfrac{1}{2} m \vert \dot{x}\vert ^2 -\left(mGM + \dfrac{4EGM}{c^2}\right) \dfrac{1}{\vert x\vert }-\dfrac{3G^2 M^2}{c^2} \dfrac{1}{\vert x\vert ^2} - \varepsilon \, U(x)=E.
\end{cases}
\end{equation}
Periodic solutions of this problem can be directly obtained from an application of Theorem~\ref{teo-levicivita}, provided the parameter $\kappa (E,c)$ given in \eqref{eq-kappalambda} is positive, i.e.~when $E>-mc^2/4$. Indeed, fixed any energy $E$ in this range, Theorem~\ref{teo-levicivita} yields solutions of any energy $E' < 0$ and so, in particular, of energy $E$. We then have the following result.

\begin{corollary}\label{cor-levi}
Let $E\in (-mc^2/4,0)$ and $N\in \mathbb{N}$. Then, for every $\varsigma>0$ there exists $\varepsilon^{*}=\varepsilon^{*}(N,\varsigma)>0$ such that for every $\varepsilon\in \mathbb{R}$, with $\vert \varepsilon\vert \leq \varepsilon^{*}$, problem \eqref{eq-LCfinale} has at least $2N$ periodic solutions $x^i_{1},\ldots, x^i_N$, $i=1,2$.

Moreover, for every $j=1,\ldots, N$ and $i=1,2$, the solution $x^i_j$ has minimal period $T_{i,j}$ satisfying
\begin{equation*}
\left\vert T_{i,j}-\dfrac{\pi \kappa (E,c)\sqrt{m}}{\sqrt{2} (-E)^{3/2}}\right\vert <\varsigma
\end{equation*}
and winding number in its minimal period
\begin{equation*}
\left\lfloor\dfrac{\sqrt{-2E}}{\kappa (E,c)}\, \sqrt{2\lambda(E,c) +\dfrac{\kappa (E,c)^2}{-2E}}\right\rfloor+j.
\end{equation*}
\end{corollary}

\subsection{The Kepler problem with relativistic differential operator} \label{subsec-keplero}

In this section we deal with the problem
\begin{equation}\label{eq-kepleropert}
\dfrac{\mathrm{d}}{\mathrm{d}t}\left(\dfrac{m\dot{x}}{\sqrt{1-\vert \dot{x}\vert ^2/c^2}}\right)=-\alpha\, \dfrac{x}{\vert x\vert ^3}+\varepsilon \, \nabla U(x), \qquad x \in \mathbb{R}^2 \setminus \{0\},
\end{equation}
where $m>0$, $\alpha>0$, $U \in \mathcal{C}^{\infty}(\mathbb{R}^2,\mathbb{R})$, and $\varepsilon \in \mathbb{R}$. Notice that in the physical interpretation of \eqref{eq-kepleropert} the constant $\alpha$ is equals to $mGM$, where $M$ is again the mass of the attracting body.

We first recall that for solutions of \eqref{eq-kepleropert}
an energy conservation law is fulfilled: indeed, every solution $x$ of \eqref{eq-kepleropert} satisfies
\begin{equation} \label{eq-energiakeplero}
E_\varepsilon(x,\dot{x})=mc^2 \left( \frac{1}{\sqrt{1-\vert \dot{x}\vert ^2/c^2}}-1\right)-\dfrac{\alpha}{\vert x\vert }-\varepsilon \, U(x)=H,
\end{equation}
for some $H\in \mathbb{R}$. 

Now, we are in a position to state and prove our result for \eqref{eq-kepleropert}.

\begin{theorem}\label{teo-keplero}
Let $H^{*}\in (-mc^2,0)$ and $N\in \mathbb{N}$. Then, for every $\varsigma>0$ there exists $\varepsilon^{*}=\varepsilon^{*}(N,\varsigma)>0$ such that for every $\varepsilon\in \mathbb{R}$, with $\vert \varepsilon\vert \leq \varepsilon^{*}$, equation \eqref{eq-kepleropert} has at least $2N$ periodic solutions $x^i_{1},\ldots, x^i_N$, $i=1,2$,	of energy $H^{*}$.

Moreover, for every $j=1,\ldots, N$ and $i=1,2$, the solution $x^i_j$ has minimal period $T_{i,j}$ satisfying
\begin{equation*}
\left\vert T_{i,j}-\dfrac{2\pi m^2 c^3}{(-2 mc^2 H^{*}-(H^{*})^2)^{3/2}}\right\vert <\varsigma
\end{equation*}
and winding number $\lfloor mc^2/(H^{*}+mc^2) \rfloor+j$ in its minimal period.
\end{theorem}

\begin{proof}
The proof is an application of Theorem~\ref{teo-weinstein} and it is based on the discussion and results in \cite{BoDaFe-21}. Please notice that in \cite{BoDaFe-21} the energy was defined by
\begin{equation*}
\dfrac{mc^2}{\sqrt{1-\vert \dot{x}\vert ^2/c^2}}-\dfrac{\alpha}{\vert x\vert }-\varepsilon \, U(x),
\end{equation*}
thus implying that it differs from the energy $E_\epsilon$ defined in \eqref{eq-energiakeplero} by the additive factor $mc^2$. As a consequence, the formulas in \cite{BoDaFe-21} involving the energy have to be modified taking into account this additive factor.

For the proof, we adopt the same strategy of the proof of Theorem~\ref{teo-levicivita} and we outline the various steps, by pointing out the possible differences.

We first recall (cf.~\cite[Section~2]{BoDaFe-21}) that \eqref{eq-kepleropert} can be written as a Hamiltonian system, with respect to the variables $(x,p)$, where 
\begin{equation*}
p=\dfrac{m\dot{x}}{\sqrt{1-\dfrac{\vert \dot{x}\vert ^2}{c^2}}}.
\end{equation*}
The Hamiltonian is given by
\begin{equation*}
\mathcal{H}_\varepsilon(x,p)=mc^2 \left( \sqrt{1+\dfrac{\vert p\vert ^2}{m^2c^2}} -1 \right)-\dfrac{\alpha}{\vert x\vert }-\varepsilon \, U(x)
\end{equation*}
and the angular momentum $\mathcal{L}_{0}$ is again defined by \eqref{eq-momanglevi}. 

Action-angles variables are defined as in the previous case, first introducing the change of variables 
\begin{equation*}
\Psi \colon \Omega\to (\mathbb{R}^2\setminus\{0\})\times \mathbb{R}^2,
\qquad
\Psi (r,\vartheta,l,\Phi)=\left(re^{i\vartheta},le^{i\vartheta}+\dfrac{\Phi}{r}ie^{i\vartheta}\right),
\end{equation*}
being $\Omega=(0,+\infty)\times \mathbb{T}^1 \times \mathbb{R}^2$.
Then, defining
\begin{equation*}
\Lambda =\left\{(H,L)\in \mathbb{R}^{2} \colon -mc^2<H<0, \, \dfrac{\alpha^2}{c^2} < L^2 < \dfrac{\alpha^2 m^2 c^2}{-2m c^2 H-H^2} \right\},
\end{equation*}
for every $(H,L)\in \Lambda$ the orbit in the $(r,l)$-plane is closed. We point out that in this case the radial period map and the apsidal angle are given by
\begin{equation} \label{eq-defperiodoK}
T(H,L)=\dfrac{2\pi m^2 c^3}{(-2m c^2 H-H^2)^{3/2}},
\end{equation}
and
\begin{equation} \label{eq-defthetaK}
\Theta (H,L)=\dfrac{2\pi}{\sqrt{1-\dfrac{\alpha^2}{c^2L^2}}},
\end{equation}
for every $(H,L)\in \Lambda$, respectively.

Proceeding as in the case of the Levi-Civita potential, we introduce action-angles variables $(I_{1},I_{2},\varphi_{1},\varphi_{2})$ and in \cite[Section~2.4]{BoDaFe-21} it is proved that in the new variables the Hamiltonian corresponding to the unperturbed case $\varepsilon=0$ is given by 
\begin{equation}\label{eq-hamactang}
\mathcal{K}_{0}(I_{1},I_{2}) = mc^2 \dfrac{I_{1} - I_{2} + \dfrac{1}{c}\sqrt{c^2I_{2}^2 - \alpha^2}}{\sqrt{(I_{1}-I_{2})^2 + I_{2}^2+\dfrac{2}{c}(I_{1}-I_{2}) \sqrt{c^2 I_{2}^2 - \alpha^2}}}-mc^2.
\end{equation}

Now, let us fix $H^{*}\in (-mc^2,0)$, $N\in \mathbb{N}$, $j\in \{1,\ldots N\}$ and let us set $k^{*}_j=\lfloor mc^2/(H^{*}+mc^2)\rfloor+j$. Then, there exists a unique $L^{*}_j>0$ such that
\begin{equation} \label{eq-pr301}
\sqrt{1-\dfrac{\alpha^2}{c^2(L^{*}_j)^2}}=\dfrac{1}{k^{*}_j}.
\end{equation}
A simple computation shows that $(H^{*},L^{*}_j)\in \Lambda$. Moreover, recalling \eqref{eq-defthetaK}, condition \eqref{eq-pr301} implies that
\begin{equation*}
\Theta (H^{*},L^{*})=2\pi  k^{*}_j
\end{equation*}
and then the solution of the unperturbed problem corresponding to the pair $(H^{*},L^{*}_j)$ is periodic of period $2\pi m^2 c^3/(-2 mc^2 H^{*}-(H^{*})^2)^{\frac{3}{2}}$ (recall \eqref{eq-defperiodoK}) and it has winding number $k^{*}_j$ in its period. 

Let us now denote by $I^{*}_j$ the pair of actions associated to $(H^{*},L^{*}_j)$ via the map $\Pi$ defined in \eqref{eq-defaaLC}. The validity of \eqref{eq-kamiso}, with $\mathcal{K}_{0}$ as in \eqref{eq-hamactang}, follows from the fact that $\nabla^2 \mathcal{K}_{0} (I^{*}_j)$ is negative defined (cf.~\cite[Section~3.2]{BoDaFe-21}).

We then conclude as in the proof of Theorem~\ref{teo-levicivita}.
\end{proof}

\begin{remark}\label{rem-3.2}
Let us notice that, analogously as in \cite{BoDaFe-21}, one could also consider values $(H^{*},L^{*})$ giving rise to periodic solutions of the unperturbed problem with minimal period $nT(H^{*},L^{*})$ with $n>1$ (this happens if and only if $\Theta(H^{*},L^{*}) \in 2\pi (\mathbb{Q} \setminus \mathbb{Z})$).
The non-degeneracy condition can be verified in a similar manner and thus a more general version of Theorem~\ref{teo-keplero} could be proved. The same remark is also valid for the Levi-Civita problem discussed in Section~\ref{subsec-levicivita}.
\hfill$\lhd$
\end{remark}

\bibliographystyle{elsart-num-sort}
\bibliography{BoDaFe-biblio}

\begin{thebibliography}{10}
\expandafter\ifx\csname url\endcsname\relax
  \def\url#1{\texttt{#1}}\fi
\expandafter\ifx\csname urlprefix\endcsname\relax\def\urlprefix{URL }\fi

\bibitem{AmBe-90}
A.~Ambrosetti, U.~Bessi, Multiple periodic trajectories in a relativistic
  gravitational field, in: Variational methods ({P}aris, 1988), vol.~4 of
  Progr. Nonlinear Differential Equations Appl., Birkh\"{a}user Boston, Boston,
  MA, 1990, pp. 373--381.

\bibitem{AnBa-71}
C.~M. Andersen, H.~C. von Baeyer, On classical scalar field theories and the
  relativistic {K}epler problem, Ann. Physics 62 (1971) 120--134.

\bibitem{Ar-89}
V.~I. Arnol'd, Mathematical methods of classical mechanics, vol.~60 of Graduate
  Texts in Mathematics, 2nd ed., Springer-Verlag, New York, 1989.

\bibitem{BeFa-notes}
G.~Benettin, F.~Fass\`{o}, {I}ntroduzione alla teoria delle perturbazioni per
  sistemi {H}amiltoniani, {L}ecture notes,
  \url{https://www.math.unipd.it/~benettin/postscript-pdf/hamilt.pdf}
  (2001-2002).

\bibitem{BoDaFe-pp}
A.~Boscaggin, W.~Dambrosio, G.~Feltrin, Periodic perturbations of central force
  problems and an application to a restricted $3$-body problem,
  arXiv:2110.11635 (2021).

\bibitem{BoDaFe-21}
A.~Boscaggin, W.~Dambrosio, G.~Feltrin, Periodic solutions to a perturbed
  relativistic {K}epler problem, SIAM J. Math. Anal. 53 (2021) 5813--5834.

\bibitem{BoDaPa-pp}
A.~Boscaggin, W.~Dambrosio, D.~Papini, Infinitely many periodic solutions to a
  {L}orentz force equation with singular electromagnetic potential,
  arXiv:2302.06189 (2023).

\bibitem{BoDaPa-23}
A.~Boscaggin, W.~Dambrosio, D.~Papini, Periodic solutions to relativistic
  {K}epler problems: a variational approach, Ann. Sc. Norm. Super. Pisa Cl.
  Sci.

\bibitem{Bo-04}
T.~H. Boyer, Unfamiliar trajectories for a relativistic particle in a {K}epler
  or {C}oulomb potential, Amer. J. Phys. 72 (2004) 992--997.

\bibitem{BrHu-91}
H.~W. Broer, G.~B. Huitema, A proof of the isoenergetic {KAM}-theorem from the
  ``ordinary'' one, J. Differential Equations 90 (1991) 52--60.

\bibitem{Co-03}
B.~Cordani, The {K}epler problem. Group theoretical aspects, regularization and
  quantization, with application to the study of perturbations, vol.~29 of
  Progress in Mathematical Physics, Birkh\"{a}user Verlag, Basel, 2003.

\bibitem{FoGa-17}
A.~Fonda, A.~C. Gallo, Radial periodic perturbations of the {K}epler problem,
  Celestial Mech. Dynam. Astronom. 129 (2017) 257--268.

\bibitem{GoPoSa-02}
H.~Goldstein, C.~Poole, J.~Safko, Classical mechanics, Addison Wesley, San
  Francisco, 2002.

\bibitem{Ji-13}
L.~Jia, Approximate {K}epler’s elliptic orbits with the relativistic effects,
  Int. J. Astron. Astrophys. 3 (2013) 29--33.

\bibitem{LaLlNu-91}
E.~Lacomba, J.~Llibre, A.~Nunes, Invariant tori and cylinders for a class of
  perturbed {H}amiltonian systems, in: The geometry of {H}amiltonian systems
  ({B}erkeley, {CA}, 1989), vol.~22 of Math. Sci. Res. Inst. Publ., Springer,
  New York, 1991, pp. 373--385.

\bibitem{LeMo-PP}
T.~J. Lemmon, A.~R. Mondragon, Kepler’s orbits and {S}pecial {R}elativity in
  {I}ntroductory {C}lassical {M}echanics, arXiv:1012.5438.

\bibitem{LC-28}
T.~Levi-Civita, Fondamenti di Meccanica Relativistica, Zanichelli, Bologna,
  1928.

\bibitem{Ma-13}
J.~Mawhin, Resonance problems for some non-autonomous ordinary differential
  equations, in: Stability and bifurcation theory for non-autonomous
  differential equations, vol. 2065 of Lecture Notes in Math., Springer,
  Heidelberg, 2013, pp. 103--184.

\bibitem{MuPa-06}
G.~Mu\~{n}oz, I.~Pavic, A {H}amilton-like vector for the special-relativistic
  {C}oulomb problem, European J. Phys. 27 (2006) 1007--1018.

\bibitem{NuCaLl-91}
A.~Nunes, J.~Casasayas, J.~Llibre, A perturbation of the relativistic {K}epler
  problem, in: Predictability, stability, and chaos in {$N$}-body dynamical
  systems ({C}ortina d'{A}mpezzo, 1990), vol. 272 of NATO Adv. Sci. Inst. Ser.
  B: Phys., Plenum, New York, 1991, pp. 547--554.

\bibitem{Po-76}
H.~Pollard, Celestial mechanics, vol.~18 of Carus Mathematical Monographs,
  Mathematical Association of America, Washington, DC, 1976.

\bibitem{ToUrZa-13}
P.~J. Torres, A.~J. Ure\~{n}a, M.~Zamora, Periodic and quasi-periodic motions
  of a relativistic particle under a central force field, Bull. Lond. Math.
  Soc. 45 (2013) 140--152.

\bibitem{We-73}
A.~Weinstein, Normal modes for nonlinear {H}amiltonian systems, Invent. Math.
  20 (1973) 47--57.

\bibitem{We-77}
A.~Weinstein, Symplectic {$V$}-manifolds, periodic orbits of {H}amiltonian
  systems, and the volume of certain {R}iemannian manifolds, Comm. Pure Appl.
  Math. 30 (1977) 265--271.

\bibitem{We-78}
A.~Weinstein, Bifurcations and {H}amilton's principle, Math. Z. 159 (1978)
  235--248.

\bibitem{Za-13}
M.~Zamora, New periodic and quasi-periodic motions of a relativistic particle
  under a planar central force field with applications to scalar boundary
  periodic problems, Electron. J. Qual. Theory Differ. Equ. (2013) No. 31, 16
  pp.

\end{thebibliography}

\end{document}